\documentclass[11pt]{article}
\usepackage{amsmath,amssymb,amsfonts,textcomp,stmaryrd,amsthm}
\usepackage{color}
\usepackage[usenames,dvipsnames,svgnames,table]{xcolor}
\usepackage{epsfig}

\newcommand{\<}  {\langle}
\renewcommand{\>}{\rangle}

\newcommand{\tH}{\tilde{H}} 
\newcommand{\R}{\ensuremath{\mathbb{R}}}
\newcommand{\grad}{\ensuremath{\mathrm{grad\,}}}
\renewcommand{\div}{\ensuremath{\mathrm{div\,}}}

\newcommand{\bsigma}{{\text{\boldmath $\sigma$}}}
\newcommand{\btau}{{\text{\boldmath $\tau$}}}

\newcommand{\bfeta}{{\text{\boldmath $\eta$}}}

\newcommand{\CT}{\mathcal{T}}

\providecommand{\BP}{{\boldsymbol{P}}}
\providecommand{\BL}{{\boldsymbol{L}}}
\providecommand{\Bn}{{\boldsymbol{n}}}

\newcommand{\normHdiv}{{\mathrm{div},\Omega}}
\newcommand{\normHdivh}{{\mathrm{div},\CT}}
\newcommand{\normHdivK}{{\mathrm{div},K}}
\newcommand{\normHdivKp}{{\mathrm{div},K'}}

\newcommand{\Hdiv}{\boldsymbol{H}(\mathrm{div},\Omega)}
\newcommand{\HdivK}{\boldsymbol{H}(\mathrm{div},K)}

\newcommand{\Hdivh}{\boldsymbol{H}(\mathrm{div},\CT)}
\newcommand{\CHdivK}{\boldsymbol{\mathcal{H}}(\mathrm{div},K)}
\newcommand{\CHdivh}{\boldsymbol{\mathcal{H}}(\mathrm{div},\CT)}
\newcommand{\normCHdivh}{{\mathrm{div},\CT,\Bn}}
\newcommand{\CV}{\mathcal{V}}
\newcommand{\CE}{\mathcal{E}}

\newtheorem{theorem}{Theorem}
\newtheorem{lemma}[theorem]{Lemma}
\newtheorem{cor}[theorem]{Corollary}

\theoremstyle{definition} 
\newtheorem{remark}[theorem]{Remark}

\title{Note on discontinuous trace approximation\\ in the practical DPG method
\thanks{Supported by CONICYT through FONDECYT projects 1110324, 3140614
        and Anillo ACT1118 (ANANUM), and by NSF through grant DMS 1216356.}}
\date{}
\author{
Norbert Heuer$^\dagger$
\and
Michael Karkulik\thanks{
Facultad de Matem\'aticas, Pontificia Universidad Cat\'olica de Chile,
Avenida Vicu\~na Mackenna 4860, Macul, Santiago, Chile,
email: {\tt \{nheuer,mkarkulik\}@mat.puc.cl}}
\and
Francisco-Javier Sayas \thanks{
Department of Mathematical Sciences, University of Delaware,
Newark DE 19716, USA, e-mail: {\tt fjsayas@math.udel.edu}}}

\begin{document}
\maketitle
\begin{abstract}
We analyze a non-conforming DPG method with discontinuous trace approximation
for the Poisson problem in two and three space dimensions. We show its
well-posedness and quasi-optimal convergence in the principal unknown.
Numerical experiments confirming the theory have been presented previously.
\end{abstract}

\section{Introduction}

Most variants of the DPG method with optimal test functions are based upon an
ultra-weak formulation of the problem under consideration. This means that
principal unknown functions (like displacements and stresses) are approximated
in weaker norms than usual so that discontinuous basis functions are conforming,
cf., e.g., \cite{DemkowiczG_11_CDP,DemkowiczG_11_ADM,DemkowiczGN_12_CDP,
DemkowiczH_13_RDM}.
The use of discontinuous approximations has obvious advantages when considering
non-uniform meshes and local mesh refinements. Standard DPG theory for
boundary value problems of second order requires, however,
that approximations of traces of primal unknowns be continuous. This is due to
the fact that, although primal unknowns are measured in $L^2$-spaces, their
traces are analyzed in trace spaces of $H^1$-functions, cf.~\cite{DemkowiczG_11_ADM}.

The numerical results presented by Demkowicz and Gopalakrishnan in
\cite{DemkowiczG_11_ADM} use conforming approximations, thus need continuous
basis functions to approximate the trace $\hat u$ of the principal unknown $u$.
However, in the preprint version \cite{DemkowiczG_10_ADM}, the authors report
on experiments where this trace unknown is approximated by discontinuous basis
functions, and no negative effects were observed.
In this note we prove that this variational crime has no effect
on the approximation of $u$ as long as the polynomial degrees stay bounded.
Our analysis can be extended to arbitrary polynomial degrees
except for the stability of the Raviart-Thomas interpolation operator.
Indeed, Lemma~\ref{la_bound} below, stating the boundedness of the bilinear form,
can be extended to arbitrary degrees. Then the trace component of the upper bound
has a logarithmical factor in $p$.

There is one catch when analyzing the DPG method (for the Poisson problem or
others of second order) with discontinuous trace approximation. The resulting
standard discrete formulation, with test functions given exactly via the so-called
trial-to-test operator mapping to the full test space, is not well posed.
This is due to the fact that the appearing duality for traces is not well defined
in the general setting. This is mentioned below and is due to the fact the local
normal traces of H(div)-functions are not dual to (local)
$H^{1/2}$-functions. Therefore, one is forced to consider discrete test functions
where dualities are well defined as $L^2$-bilinear forms. This amounts to an
approximation of the trial-to-test operator and has been introduced
by Gopalakrishnan and Qiu as the practical DPG method, see \cite{GopalakrishnanQ_APD}.
In practice, the trial-to-test operator is approximated anyway and therefore,
considering our non-conforming variant only for the practical DPG method is both
necessary for theoretical reasons and relevant.

An overview of this note is as follows. In the next section we recall the model
problem and its standard ultra-weak formulation, and present our non-conforming discrete
scheme. Furthermore, we state the main result (Theorem~\ref{thm}) which shows that,
for the Poisson problem, the DPG method with discontinuous trace approximation
is well posed and that the $L^2$-error in $u$ converges quasi-optimally with respect
to mesh refinement. Technical results and a proof of the theorem are given in
\S\ref{sec_proofs}.

Throughout the paper, $a\lesssim b$ means that $a\le cb$ with a generic constant $c>0$
that is independent of involved mesh parameters or functions. Similarly, the notation
$a\gtrsim b$ and $a\simeq b$ is used.

\section{Non-conforming DPG method and main result}

Let us first introduce the Poisson model problem. We consider a bounded simply
connect Lipschitz domain $\Omega\subset\R^n$ ($n=2,3$), which is assumed to
be polygonal in two dimensions and polyhedral otherwise. The boundary of $\Omega$
is denoted by $\partial\Omega$. For given $f\in L^2(\Omega)$ we consider
\begin{equation} \label{prob}
   -\Delta u  = f \ \mbox{in}\ \Omega,\quad u = 0 \ \mbox{on}\ \partial\Omega.
\end{equation}

\paragraph{Ultra-weak formulation.}
To present an ultra-weak formulation of \eqref{prob} we consider a (sequence of)
mesh(es) $\CT$ of shape-regular simplexes $\CT=\{K\}$ and allow for hanging nodes/edges.
We assume that there is a finite number of patterns where a face is a strict subset of
a neighbor's face. This can be achieved, e.g., by starting with a coarse mesh and
using standard mesh refinement strategies.
The diameter of $K\in\CT$ is $h_K$ and we assume that the ratio of diameters of
neighboring elements is bounded. The skeleton is denoted by
$\Gamma=\cup_{K\in\CT}\partial K$ and we define $\Gamma_0:=\Gamma\cap\Omega$.

In the following we will refer to the edges of elements $K\subset\R^2$ also as faces.
We consider a set of non-overlapping faces $G=\{g\}$ of the mesh with:\\
(i) $G$ covers $\Gamma$: $\Gamma=\cup_{g\in G}\bar g$,\quad
(ii) each $g\in G$ is an entire face of an element $K\in\CT$.\\
We assign to each $g\in G$ exactly one element $K\in\CT$ which has
$g$ as an entire face. This assignment will be denoted by $g\in G\mapsto K(g)\in\CT$.
We also define the subset $G_0:=\{g\in G;\; g\subset\Gamma_0\}$.

We use standard (scalar) $L^2(\Omega)$, (vector) $\BL^2(\Omega)$, (vector) $\Hdiv$, and
$H^1(\Omega)$ spaces with $L^2(\Omega)$, $\BL^2(\Omega)$, $\Hdiv$  and
$H^1(\Omega)$-norms denoted by $\|\cdot\|$, $\|\cdot\|_\normHdiv$ and $\|\cdot\|_{1,\Omega}$,
respectively. We we will also need the broken spaces
\begin{align*}
   H^1(\CT)  &:= \{v\in L^2(\Omega);\; v|_K\in H^1(K)\ \forall K\in\CT\},\\
   \Hdivh  &:= \{\btau\in\BL^2(\Omega);\; \btau|_K\in\HdivK\ \forall K\in\CT\}
\end{align*}
with respective product norms
$\|\cdot\|_{1,\CT} = (\sum_{K\in\CT} \|\cdot\|_{1,K}^2)^{1/2}$,
$\|\cdot\|_\normHdivh = (\sum_{K\in\CT} \|\cdot\|_\normHdivK^2)^{1/2}$.
Below, we also need the $L^2(K)$-norm denoted by $\|\cdot\|_K$.
Furthermore we consider the trace spaces
\begin{align*}
   \tH^{1/2}(\Gamma_0)  &:= H^1_0(\Omega)|_{\Gamma_0},\\
   H^{-1/2}(\Gamma)  &:=
   \{\eta;\; \exists \btau\in\Hdiv:\; (\btau\cdot\Bn)|_\Gamma=\eta\}
\end{align*}
with trace norms
\begin{align*}
   \|v\|_{1/2,\Gamma_0} &:=
   \inf_{w\in H^1_0(\Omega);\; w|_{\Gamma_0} = v} \|w\|_{1,\Omega},\\
   \|\eta\|_{-1/2,\Gamma} &:=
   \inf_{\btau\in\Hdiv,\; (\btau\cdot\Bn)|_\Gamma = \eta} \|\btau\|_{\normHdiv}.
\end{align*}
Here,
$\Bn$ is a unit vector field on $\Gamma$ being normal to the faces in a certain
direction and pointing outwards on $\partial\Omega$.

Introducing $\bsigma:=\grad u$, $\hat u:=u|_{\Gamma_0}$, and
$\hat\sigma:=\bsigma\cdot\Bn|_\Gamma$, an ultra-weak formulation of \eqref{prob} reads:
{\em Find $(u,\bsigma,\hat u,\hat\sigma)\in U^{\mathrm{conf}} :=
L^2(\Omega)\times\BL^2(\Omega)\times\tH^{1/2}(\Gamma_0)\times H^{-1/2}(\Gamma)$ such that}
\begin{equation} \label{weak}
\begin{split}
   (\bsigma, \btau)_\CT + (u, \div\btau)_\CT
   - \<[\btau\cdot\Bn], \hat u\>_{\Gamma_0}
   &= 0
   \ \qquad\qquad\forall \btau\in\Hdivh, \\
   (\bsigma, \grad v)_{\CT} - \<\hat\sigma, [v]\>_\Gamma &= (f, v)_\Omega
   \qquad\forall v\in H^1(\CT).
\end{split}
\end{equation}
Here, $(\cdot,\cdot)_\Omega$, $\<\cdot,\cdot\>_{\Gamma_0}$ and $\<\cdot,\cdot\>_\Gamma$
denote dualities which extend the $L^2$-products on $\Omega$, $\Gamma_0$, and $\Gamma$,
respectively. The notation $(\cdot,\cdot)_\CT$ also denotes dualities on $\Omega$ but
indicates that appearing differential operators are calculated piecewise with respect
to the mesh $\CT$. Furthermore, $[v]$ and $[\btau\cdot\Bn]$ denote the jumps
of $v$ and $\btau\cdot\Bn$, respectively, across $\Gamma_0$, defined to be consistent
with the normal vector field $\Bn$, and on $\partial\Omega$,
$[v]$ reduces to the trace of $v$.

The left-hand side of system \eqref{weak} is the DPG bilinear form,
\begin{equation} \label{bfg}
   b((u,\bsigma,\hat u,\hat\sigma),(v,\btau))
   :=
   (\bsigma, \btau+\grad v)_\CT
   +
   (u, \div\btau)_\CT
   -
   \<[\btau\cdot\Bn], \hat u\>_{\Gamma_0} - \<\hat\sigma, [v]\>_\Gamma,
\end{equation}
and the test space of the ultra-weak formulation is
\[
   V := H^1(\CT)\times \Hdivh
   \quad\text{with norm}\quad
   \|(v,\btau)\|_V^2 := \|v\|_{1,\CT}^2 + \|\btau\|_{\normHdivh}^2
\]
and corresponding inner product $\<\cdot,\cdot\>_V$.
In \cite{DemkowiczG_11_ADM}, Demkowicz and Gopalakrishnan show that \eqref{weak}
has a unique solution. Furthermore, they propose a conforming discretization and
show its quasi-optimal convergence in the norm
\begin{equation} \label{norm_Uconf}
   \|(u,\bsigma,\hat u,\hat\sigma)\|_{U^\mathrm{conf}}^2
   =
   \|u\|^2 + \|\bsigma\|^2 + \|\hat u\|_{1/2,\Gamma_0}^2 + \|\hat\sigma\|_{-1/2,\Gamma}^2,
   \quad (u,\bsigma,\hat u,\hat\sigma)\in U^\mathrm{conf}.
\end{equation}
Conformity requires continuity of basis function only for the $\hat u$-component.
The other components are measured in low-regular spaces which can be approximated by
piecewise polynomials without continuity restriction.
In this note, we present a non-conforming variant that uses discontinuous piecewise
polynomial basis function for the approximation of the trace $\hat u$, and prove
its quasi-optimal convergence for the component $u$. In this way, simple discontinuous
basis functions can be used throughout.

\paragraph{Discretization.}
Let $P_p(K)$ be the space of polynomials up to degree $p$ on $K\in\CT$,
and let $\BP_p(K)$ be the space of vector fields of polynomials with components in $P_p(K)$.
Furthermore, $P_p(\partial K)$ is the space of functions which are polynomials of degree up
to $p$ on each face of $K$. Correspondingly, we define piecewise polynomial spaces
$P_p(\CT)$, $\BP_p(\CT)$, $P_p(G)$, and $P_p(G_0)$ without any inter-element continuity.

The approximation space for the non-conforming DPG method is
\begin{align*}
   U_h &:= U_h^u\times U_h^\bsigma\times U_h^{\hat u}\times U_h^{\hat\sigma}
   \quad\text{with}\\
   U_h^u        &:= P_p(\CT),\quad     U_h^\bsigma := \BP_p(\CT),\quad
   U_h^{\hat u}  := P_{p+1}(G_0),\quad U_h^{\hat\sigma} := P_p(G).
\end{align*}
For integer $r\ge p+n$, we define the discrete test space
\begin{align*}
   V^r = P_r(\CT)\times\BP_{p+2}(\CT).
\end{align*}
Now, the practical DPG method with discontinuous trace approximation reads as follows.
{\em Find $\phi_h:=(u_h,\bsigma_h,\hat u_h,\hat\sigma_h)\in U_h$ such that}
\begin{equation} \label{DPG}
   b(\phi_h, (v,\btau)) = (f,v)_\Omega
   \quad\forall (v,\btau)\in T^r(U_h)
\end{equation}
with trial-to-test operator
\[
\begin{split}
   T^r:\; U: =   L^2(\Omega)\times\BL^2(\Omega)\times H^{1/2}(G_0)\times H^{-1/2}(\Gamma)
             \to V^r:\\
   \<T^r(\phi), (v,\btau)\>_V = b(\phi, (v,\btau))
   \quad\forall (v,\btau)\in V^r.
\end{split}
\]
Here, $H^{1/2}(G_0)$ is the product (or broken) space
\[
   H^{1/2}(G_0) = \Pi_{g\in G_0} H^{1/2}(g)
   \quad\text{with}\quad
   H^{1/2}(g) := H^1(\Omega)|_g\quad\forall g\in G_0.
\]
In order to facilitate deriving an appropriate a priori error estimate for our
DPG approximation we furnish $H^{1/2}(G_0)$ with a scalable norm $\|\cdot\|_{1/2,G_0}$.
To this end, we define for any $g\in G$ the (semi-)norms
\begin{equation} \label{3}
   |v|_{1/2,g}^2 :=\int_g\int_g \frac{|v(x)-v(y)|^2}{|x-y|^{n}} \,ds_x\,ds_y,
   \quad
   \|v\|_{1/2,g}^2 := h_{K(g)}^{-1}\|v\|_{L^2(g)}^2 + |v|_{H^{1/2}(g)}^2,
\end{equation}
and define
\[
   \|v\|_{1/2,G_0}^2 := \sum_{g\in G_0} \|v\|_{1/2,g}^2,
   \quad v\in H^{1/2}(G_0).
\]
Note that the bilinear form $b$ is not well defined on $U\times V$ because, in general,
$\btau\cdot\Bn|_g\not\in\tilde H^{-1/2}(g)=(H^{1/2}(g))'$ for $(v,\btau)\in V$
and $g\in G_0$ (one does have the weaker regularity
$\btau\cdot\Bn|_{\partial K}\in H^{-1/2}(\partial K)$ for $K\in\CT$).
Furthermore, since $U_h\not\subset U^\mathrm{conf}$, \eqref{DPG} is a
non-conforming discretization of \eqref{weak}.
We also note that the discrete scheme \eqref{DPG} is consistent since
the the discrete test space is conforming, $T^r(U_h)\subset V^r\subset V$.
  
Having a new norm for the $\hat u$-components at hand, we define a $U$-norm by
\[
   \|(u,\bsigma,\hat u,\hat\sigma)\|_U^2
   =
   \|u\|^2 + \|\bsigma\|^2 + \|\hat u\|_{1/2,G_0}^2
           + \|\hat\sigma\|_{-1/2,\Gamma}^2,
   \quad (u,\bsigma,\hat u,\hat\sigma)\in U.
\]
This norm defines the upper bound of our a priori error estimate for the DPG method
whereas the error itself considers only the approximation of $u$.

\paragraph{Main result.}
Our main result is the following error estimate. It is quasi-optimal for bounded $p$.

\begin{theorem} \label{thm}
Denote by $(u,\bsigma,\hat u,\hat\sigma)\in U^\mathrm{conf}$ the solution
of \eqref{weak}. If $r\ge p+n$ then the discrete scheme \eqref{DPG} has a
unique solution $(u_h,\bsigma_h,\hat u_h,\hat\sigma_h)\in U_h$ and there holds
the quasi-optimal error estimate
\[
   \|u-u_h\|
   \lesssim
   \inf_{(u_d,\bsigma_d,\hat u_d,\hat\sigma_d)\in U_h}
   \|(u-u_d,\bsigma-\bsigma_d,\hat u-\hat u_d,\hat\sigma-\hat\sigma_d)\|_U.
\]
\end{theorem}

A proof of this theorem is given in \S\ref{sec_pf}.

By standard approximation results (cf., e.g., the details given in
\cite{GopalakrishnanQ_APD}) we obtain the following convergence estimate.
It has the typical order of a DPG method for the chosen ultra-weak formulation.
Note that the order of convergence is reduced by the regularity of $\hat u$
which is bounded, via a trace argument, by the regularity of $u$.

\begin{cor}
For $r\ge p+n$, and assuming sufficient regularity of the solution
$\phi=(u,\bsigma,\hat u,\hat\sigma)$ of \eqref{weak}, there holds
\[
   \|u-u_h\|
   \lesssim
   h^s ( \|u\|_{H^{s+1}(\Omega)} + \|\bsigma\|_{H^{s+1}(\Omega)})
\]
for $s\in (1/2,p+1]$ and with $h=\max_{K\in\CT} h_K$.
\end{cor}

\section{Proofs} \label{sec_proofs}

We split the proof of Theorem~\ref{thm} into several parts.
First we study different types of inf-sup properties. Afterwards we analyze the
discrete boundedness of the bilinear form $b$. The final proof is given in \S\ref{sec_pf}.

\subsection{Inf-sup condition} \label{sec_infsup}

For technical reasons we introduce yet another space $\CHdivh$ as
\begin{align*}
   \CHdivh  &:= \{\btau\in\BL^2(\Omega);\; \btau|_K\in\CHdivK\ \forall K\in\CT\}\quad\text{with}\\
   \CHdivK  &:= \{\btau\in\HdivK;\; (\btau\cdot\Bn)|_{\partial K}\in L^2(\partial K)\}
                \quad\forall K\in\CT
\end{align*}
and scaled product norm
\[
   \|\btau\|_\normCHdivh^2 =
   \|\btau\|_\normHdivh^2 + \sum_{K\in\CT} h_K\|\btau\,\cdot\Bn\|_{\partial K}^2
   \quad(\btau\in\CHdivh),
\]
and consider the following subspace of $V$,
\[
   \CV := H^1(\CT) \times \CHdivh
\]
with norm $\|(v,\btau)\|_\CV^2=\|v\|_{1,\CT}^2+\|\btau\|_\normCHdivh^2$.
We will also make use of Sobolev spaces $H^s(\Omega)$ with $s\in (1/2,1)\cup (3/2,2)$, furnished
with the Sobolev-Slobodeckij norm
\[
   \|v\|_{s,\Omega}^2
   =
   \left\{\begin{array}{ll}
      \displaystyle
      \|v\|^2
      +
      \int_\Omega\int_\Omega \frac {|v(x)-v(y)|^2}{|x-y|^{n+2s}}\,dx\,dy
      &(s\in(1/2,1)),\\
      \displaystyle
      \|v\|_{1,\Omega}^2
      +
      \int_\Omega\int_\Omega \frac {|\grad v(x)-\grad v(y)|^2}{|x-y|^{n+2(s-1)}}\,dx\,dy
      &(s\in(3/2,2)).
   \end{array}\right.
\]
The notation of this norm will also be used for vector functions.
Note that the definition of this norm implies that there holds
(using the notation $\|\cdot\|_{s,\Omega}$ also for elements $K\in\CT$)
\begin{equation} \label{dd}
   \|v\|_{s,\CT} := \Bigl(\sum_{K\in\CT} \|v\|_{s,K}^2\Bigr)^{1/2}
   \le
   \|v\|_{s,\Omega}
   \quad\forall v\in H^s(\Omega),\ s\in (1/2,1)\cup(3/2,2).
\end{equation}

Essential technical detail of the proofs will be the injectivity of the bilinear form
$b$, when reducing the test space $V$ to the subspace $\CV$. This is the following lemma.

\begin{lemma} \label{la1}
The expression
\[
   \|\phi\|_{\CE}
   :=
   \sup_{\psi\in \CV\setminus\{0\}}
   \frac{b(\phi,\psi)}{\|\psi\|_\CV}
\]
is a norm in $U$.
\end{lemma}

\begin{proof}
The boundedness of the expression is immediate with a constant $C$ that may depend on the
mesh,
\[
   b(\phi,\psi)
   \le
   C(\CT)
   \|\phi\|_U \|\psi\|_\CV
   \quad\forall \phi\in U,\ \psi\in\CV.
\]
We are left to show injectivity, i.e., that there holds
\[
   \{\phi\in U;\; b(\phi,\psi)=0\ \forall\psi\in\CV\} = \{0\}.
\]
This can be shown very similarly as \cite[Lemma 4.1]{DemkowiczG_11_ADM}.
We briefly recall the steps.

For given $\phi=(u,\bsigma,\hat u,\hat\sigma)\in U$ with $b(\phi,\psi)=0$
for any $\psi\in\CV$ it follows that for any $K\in\CT$,
\[
   \div\bsigma=0,\quad \grad u=\bsigma\quad\text{on}\quad K
\]
and, by integration by parts,
\[
   \<u-\hat u^0,\btau\cdot\Bn\>_{\partial K}=0,\quad
   \<v,\hat\sigma-\bsigma\cdot\Bn\>_{\partial K}=0
\]
for any $\btau\in\CHdivK$ and $v\in H^1(K)$. Here, $\hat u^0$ is the extension
of $\hat u$ by $0$ onto $\partial\Omega$, where needed. Since
$u|_{\partial K}\in H^{1/2}(\partial K)$ for $K\in\CT$, it follows that
$\hat u^0\in H^{1/2}(\partial K)$.
Combining these results for all elements $K\in\CT$, one concludes that
$u\in H^1_0(\Omega)$ and $\bsigma\in\Hdiv$ with $\grad u=\bsigma$ and $\div\bsigma=0$
on $\Omega$. Furthermore we have that $u|_{\Gamma_0}=\hat u$ and
$\hat\sigma=(\bsigma\cdot\Bn)_\Gamma$.
Since $\Delta u=0$, we deduce that $u=0$, $\bsigma=0$, $\hat u=0$, and $\hat\sigma=0$.
This finishes the proof.
\end{proof}

We now prove a continuous inf-sup condition to control the $u$-component of
functions $\phi=(u,\bsigma,\hat u,\hat\sigma)\in U$.

\begin{lemma} \label{la_infsup_c}
There holds
\[
   \|u\|
   \lesssim
   \sup_{\psi\in \CV\setminus\{0\}}
   \frac{b(\phi,\psi)}{\|\psi\|_\CV}
   \quad\forall \phi=(u,\bsigma,\hat u,\hat\sigma)\in U.
\]
\end{lemma}

\begin{proof}
Let $\phi=(u,\bsigma,\hat u,\hat\sigma)\in U$ be given.
We define $v\in H^1_0(\Omega)$ as the
solution of $-\Delta v=u$ on $\Omega$, and set $\btau:=-\grad v$.
It follows that $\|(v,\btau)\|_V\lesssim \|u\|$. If we show the additional bound
\begin{equation} \label{add_bound}
   \sum_{K\in\CT} h_K\|\btau\,\cdot\Bn\|_{\partial K}^2 \lesssim \|u\|^2,
\end{equation}
it follows that $\|(v,\btau)\|_\CV\lesssim \|u\|$ and
\[
   \frac{b(\phi,(v,\btau))}{\|(v,\btau)\|_\CV}
   =
   \frac{ (u, u)_\Omega}{\|(v,\btau)\|_\CV}
   \gtrsim
   \|u\|.
\]
We are left to prove \eqref{add_bound}. It is well known that there exists $s\in (3/2,2)$
such that the solution $v$ of $-\Delta v=u$ satisfies $v\in H^s(\Omega)$ with
bound $\|v\|_{s,\Omega}\lesssim \|u\|$. In particular, we can bound the normal derivative
of $v$ on $\partial K$ (which equals $\pm\btau\cdot\Bn$) for any $K\in\CT$ by $\|u\|$.
We only need to check that the appearing constants do not depend on the mesh.
This can be seen by the scaling properties of the involved norms, as follows.
For a given $K\in\CT$ and reference element $\hat K$ denote transformed functions
generically with an additional $\hat\ $. Then, applying the Piola transformation
and the continuity of the normal trace $H^{s-1}(\hat K)\to L^2(\partial\hat K)$, we obtain
\begin{align*}
   h_K\|\btau\cdot\Bn\|_{\partial K}^2
   \simeq
   h_K^{2-n}\|\hat\btau\cdot\hat\Bn\|_{\partial\hat K}^2
   \lesssim
   h_K^{2-n}\|\hat\btau\|_{s-1,\hat K}^2
   \lesssim
   \|\btau\|_{s-1,K}^2.
\end{align*}
Summing over $K\in\CT$ and making use of \eqref{dd} proves that
\[
   \sum_{K\in\CT} h_K\|\btau\cdot\Bn\|_{\partial K}^2
   \lesssim
   \|\btau\|_{s-1,\Omega}^2
   =
   \|\grad v\|_{s-1,\Omega}^2
   \lesssim
   \|u\|^2.
\]
This shows \eqref{add_bound} and finishes the proof of the lemma.
\end{proof}

As an implication of Lemma~\ref{la_infsup_c} we obtain the discrete control of the
$u$-component of any discrete function $\phi\in U_h$.

\begin{lemma} \label{la_infsup}
For $r\ge p+n$ there holds
\[
   \sup_{\psi\in V^r\setminus\{0\}}
   \frac{b(\phi,\psi)}{\|\psi\|_V} \gtrsim \|u\|
   \quad\forall \phi=(u,\bsigma,\hat u,\hat\sigma)\in U_h.
\]
\end{lemma}

\begin{proof}
We follow the standard idea of using projection operators to deduce a discrete
inf-sup property from the corresponding continuous property.
More precisely, we employ the operator
\[
   \Pi^\mathrm{grad}_r:\; H^1(\CT)\to P_r(\CT)
\]
from Gopalakrishnan and Qiu \cite{GopalakrishnanQ_APD} which is continuous
in $H^1(\CT)$, and combine it (differently from \cite{GopalakrishnanQ_APD}) with
the standard piecewise Raviart-Thomas interpolation operator
\[
   \Pi^\mathrm{RT}_{p+2}:\; \CHdivh\to \BP_{p+2}(\CT)
\]
which, for bounded polynomial degree $p$, is bounded in $\CHdivh$.
The boundedness of the normal components and of the divergence is immediate
by the definition of the operator, and a bound for the $L^2$-norm can be found
in the proof of \cite[Theorem 5.36]{DiPietroE_12_MAD}.

Now, let $\phi=(u,\bsigma,\hat u,\hat\sigma)\in U_h$ be given.
Then, from the continuous inf-sup property in Lemma~\ref{la_infsup_c},
the relaxed conformity $U_h\subset U$, and the continuity of the operators just
introduced, we obtain
\begin{align*}
   \|u\|
   &\lesssim
   \sup_{(v,\btau)\in \CV\setminus\{0\}}
   \frac{b(\phi,(v,\btau))}{\|(v,\btau)\|_\CV}
   =
   \sup_{(v,\btau)\in \CV\setminus\{0\}}
   \frac{b(\phi,(\Pi^\mathrm{grad}_r v, \Pi^\mathrm{RT}_{p+2} \btau))}{\|(v,\btau)\|_\CV}
   \\
   &\lesssim
   \sup_{(v,\btau)\in \CV,\;
         (\Pi^\mathrm{grad}_r v, \Pi^\mathrm{RT}_{p+2} \btau)\not=0}
   \frac{b(\phi,(\Pi^\mathrm{grad}_r v, \Pi^\mathrm{RT}_{p+2} \btau))}
        {\|(\Pi^\mathrm{grad}_r v, \Pi^\mathrm{RT}_{p+2} \btau)\|_\CV}
   \lesssim
   \sup_{(v,\btau)\in V^r\setminus\{0\}}
   \frac{b(\phi,(v,\btau))}{\|(v,\btau)\|_\CV}.
\end{align*}
Using the trivial bound $\|\cdot\|_V\le \|\cdot\|_\CV$, the assertion follows.
\end{proof}

\begin{remark} \label{rem}
In the proof of Lemma~\ref{la_infsup} we have in particular shown that there holds
\[
   \|\phi\|_\CE
   :=
   \sup_{(v,\btau)\in \CV\setminus\{0\}}
   \frac{b(\phi,(v,\btau))}{\|(v,\btau)\|_\CV}
   \lesssim
   \sup_{(v,\btau)\in V^r\setminus\{0\}}
   \frac{b(\phi,(v,\btau))}{\|(v,\btau)\|_V}
   \quad\forall\phi\in U_h,
\]
cf.~Lemma~\ref{la1}.
\end{remark}

\subsection{Boundedness of the bilinear form} \label{sec_bound}

To analyze the jumps of $v\in H^1(\CT)$ across $\Gamma$
we define the norm of the dual space of $H^{-1/2}(\Gamma)$:
\begin{align*}
   \|[v]\|_\Gamma
   &:=
   \sup_{\eta\in H^{-1/2}(\Gamma)\setminus\{0\}} \frac {\<\eta,[v]\>}{\|\eta\|_{-1/2,\Gamma}}
   =
   \sup_{\bfeta\in\Hdiv\setminus\{0\}} \frac {\<\bfeta\cdot\Bn,[v]\>}{\|\bfeta\|_{\Hdiv}}.
\end{align*}
In contrast, the normal jumps of $\btau\in\Hdivh$ across $\Gamma$ are
controlled by a norm that is dual to our previously defined scalable intrinsic norm
of traces of $H^1(\Omega)$-functions, cf.~\eqref{3}, and, therefore, is scalable.
In other words, we define for any $g\in G$ the space
\[
   \tH^{-1/2}(g) := (H^{1/2}(g))'
\]
with corresponding norm by duality, and the broken space with respective norm
\[
   \tH^{-1/2}(G_0) := (H^{1/2}(G_0))',\quad
   \|v\|_{-1/2,G_0}^2 := \sum_{g\in G_0} \|v\|_{-1/2,g}^2.
\]

\begin{lemma} \label{la_bound}
For bounded polynomial degree $p$ there holds
\[
   b(\phi,\psi) \lesssim \|\phi\|_U \|\psi\|_V
   \quad \forall \phi\in U,\ \psi\in V^r.
\]
\end{lemma}

\begin{proof}
Let $\phi=(u,\sigma,\hat u,\hat\sigma)\in U$ and $\psi=(v,\btau)\in V^r$ be given.
We decompose $\phi=\phi_0+\phi_{\hat u}$ with $\phi_{\hat u}=(0,0,\hat u,0)$,
and bound by duality
\[
   b(\phi,\psi) = b(\phi_0,\psi) + \<[\btau\cdot\Bn],\hat u\>_{\Gamma_0}
   \lesssim
   \|\phi_0\|_U \|\psi\|_V +
   \|\hat u\|_{1/2,G_0} \|[\btau\cdot\Bn]\|_{-1/2,G_0}.
\] 
It remains to bound $\|[\btau\cdot\Bn]\|_{-1/2,G_0}$.
Consider a face $g\in G_0$. It is an entire face of $K(g)\in\CT$.
\begin{enumerate}
\item
We start by bounding $\|(\btau|_{K(g)}\cdot\Bn)\|_{-1/2,g}$.
For $\btau\in\BP_{p+2}(K(g))$ we consider the Piola transformed function
$\hat\btau$ on a reference element $\hat K$ with corresponding face $\hat g$.
Since
\[
   \|\hat\btau\cdot\hat\Bn\|_{-1/2,\hat g}
   \lesssim
   \|\hat\btau\|_{\div,\hat K}
\]
due to the equivalence of norms in finite-dimensional spaces, scaling proves
\begin{align*}
   \|\btau\cdot\Bn\|_{-1/2,g}
   \lesssim
   \bigl(\|\btau\|_K + h_K \|\div\btau\|_K\bigr)
   \lesssim
   \|\btau\|_{\normHdivK}.
\end{align*}
\item
If $g$ is an entire face also of an element $K'\in\CT$ adjacent to $K(g)$ then
we prove analogously as before that
\begin{equation} \label{2}
   \|\btau|_{K'}\cdot\Bn\|_{-1/2,g}
   \lesssim
   \|\btau\|_{\normHdivKp}.
\end{equation}
\item
If $g$ is a subset of a face of an element $K'\in\CT$ adjacent to $K(g)$ then
we proceed as follows. Applying the Piola transform to $K'$ we obtain transformed
functions as before, and $\hat g$ is a subset of a face $\hat g'$ of $\hat K'$.
There holds
\begin{align*}
   \|\hat\btau\cdot\hat\Bn\|_{-1/2,\hat g}
   &\lesssim
   \|\hat\btau\cdot\hat\Bn\|_{-1/2,\hat g'}
   \lesssim
   \|\hat\btau\|_{\div,\hat K'}.
\end{align*}
Here, we use that there are only finitely many patterns of sub-faces
$\hat g\subset\hat g'$ for all meshes $\CT$ and all faces $g$ considered in this case.
Making use of the assumption $\max\{h_k/h_{K'}, h_{K'}/h_K\}=O(1)$ and scaling then
proves the bound \eqref{2}.
\item
If $g$ is covered by faces $\{g'\}=G'$ of adjacent elements
$\{K'\}=\CT'\subset\CT\setminus\{K(g)\}$
then we apply the Piola transform to all the elements of $\CT'$, and use the notation
from before. Since there are only finitely many patterns of subsets
$\hat g|_{\hat g'}\subset \hat g'$ for each $K'\in\CT'$ mapped onto $\hat K'$,
we estimate
\begin{align*}
   \|\hat\btau\cdot\hat\Bn\|_{-1/2,\hat g}^2
   \lesssim
   \sum_{g'\in G'}
   \|\hat\btau\cdot\hat\Bn\|_{-1/2,\hat g'}^2
   &\lesssim
   \sum_{K'\in\CT'}
   \|\hat\btau\|_{\div,\hat K'}^2
\end{align*}
and obtain by scaling
\[
   \|\btau|_{\cup\{K'\in\CT'\}}\cdot\Bn\|_{-1/2,g}^2
   \lesssim
   \sum_{K'\in\CT'}
   \|\btau\|_{\normHdivKp}^2.
\]
\end{enumerate}
We finish the proof with an application of the triangle inequality and by
(squaring and) summing the estimates of the four cases with respect to $K\in\CT$.
\end{proof}

\subsection{Proof of Theorem~\ref{thm}} \label{sec_pf}

The a priori error estimate follows by standard Strang-type arguments used for
non-conforming methods. More precisely, adding and subtracting any discrete function,
making use of the consistency of \eqref{DPG}, for bounding $\|u-u_h\|$ it is enough
to have the discrete inf-sup condition
\[
   \sup_{\psi\in T^r(U_h)\setminus\{0\}}
   \frac{b(\phi,\psi)}{\|\psi\|_V} \gtrsim \|u\|
   \quad\forall \phi=(u,\bsigma,\hat u,\hat\sigma)\in U_h
\]
and the boundedness of $b$ on $U\times V^r$.
The boundedness is shown by Lemma~\ref{la_bound}, and Lemma~\ref{la_infsup}
implies the needed discrete inf-sup property:
\[
   \sup_{\psi\in V^r\setminus\{0\}} \frac{b(\phi,\psi)}{\|\psi\|_V}
   =
   \sup_{\psi\in V^r\setminus\{0\}} \frac{\<T^r\phi,\psi\>_V}{\|\psi\|_V}
   =
   \frac{\<T^r\phi,T^r\phi\>_V}{\|T^r\phi\|_V}
   =
   \sup_{\psi\in T^r(U_h)\setminus\{0\}} \frac{b(\phi,\psi)}{\|\psi\|_V}
\]
for any $\phi\in U_h$, cf.~\cite{GopalakrishnanQ_APD}.

This very argument implies that a discrete inf-sup property of $b$ on
$U_h\times T^r(U_h)$, which holds by Remark~\ref{rem}, leads to a discrete
inf-sup property of $b$ on $U_h\times V^r$.
Therefore, the square system \eqref{DPG} is uniquely solvable.
This finishes the proof of Theorem~\ref{thm}.

\bibliographystyle{siam}
\bibliography{/home/norbert/tex/bib/bib,/home/norbert/tex/bib/heuer}

\end{document}